\documentclass[12pt,reqno]{amsart}
\usepackage[latin1]{inputenc}
\usepackage[T1]{fontenc}
\usepackage{amssymb}
\usepackage{bm}

\usepackage{tikz}
\usetikzlibrary{arrows}

\usepackage[a4paper,top=3.3cm,bottom=3cm,left=3cm,right=3cm,bindingoffset=5mm]{geometry}

\usepackage[normalem]{ulem}
\newcommand{\demph}[1]{\uwave{#1}} 

\usepackage{thm-restate}

\usepackage{color}

\usepackage{paralist}

\usepackage[colorlinks=true]{hyperref}
\usepackage{cleveref}

\newtheorem{theorem}{Theorem}[section]
\newtheorem{lemma}[theorem]{Lemma}
\newtheorem{corollary}[theorem]{Corollary}
\newtheorem{proposition}[theorem]{Proposition}

\theoremstyle{definition}
\newtheorem{example}[theorem]{Example}
\newtheorem{remark}[theorem]{Remark}
\newtheorem{definition}[theorem]{Definition}
\newtheorem{question}[theorem]{Question}

\numberwithin{equation}{section}

\newcommand{\qq}[1]{``#1''}
\newcommand{\set}[1]{\left\{\,#1\,\right\}}  
\newcommand{\with}{\ \vrule\ }

\newcommand{\NN}{\mathbb{N}}
\newcommand{\RR}{\mathbb{R}}
\newcommand{\ZZ}{\mathbb{Z}}
\newcommand{\QQ}{\mathbb{Q}}
\newcommand{\kk}{\mathbb{K}}
\newcommand{\ST}{\mathcal{S}}

\newcommand{\FF}{\mathbb{F}}

\DeclareMathOperator{\dep}{depth}
\DeclareMathOperator{\pdim}{pdim}

\DeclareMathOperator{\hdep}{Hdep}
\DeclareMathOperator{\Tor}{Tor}

\newcommand{\gb}{\mathbf{g}}
\newcommand{\ab}{\mathbf{a}}
\newcommand{\bb}{\mathbf{b}}
\newcommand{\ub}{\mathbf{u}}
\newcommand{\vb}{\mathbf{v}}
\newcommand{\eb}{\mathbf{e}}
\newcommand{\db}{\mathbf{d}}
\newcommand{\mb}{\mathbf{m}}
\newcommand{\tb}{\bm{t}}
\newcommand{\Nb}{\mathbf{N}}
\newcommand{\mm}{\mathfrak{m}}
\renewcommand{\Nb}{\gb}
\newcommand{\Zb}{\bm{Z}}
\newcommand{\rb}{\mathbf{r}}

\newcommand{\HP}[1]{\mathcal{P}(#1)}

\newcommand{\STp}{\ST'}
\newcommand{\down}[1]{\mathcal{D}(#1)}

\newcommand{\coeff}[2]{\mathfrak{c}({#1}, {#2})}

\begin{document}
\title[Which series are Hilbert series?]{Which series are Hilbert series of graded modules over polynomial rings?}
\dedicatory{Dedicated to Winfried Bruns on the occasion of his 70th birthday}

\author{Lukas Katth\"an}
\address{Goethe-Universit\"at Frankfurt, FB Informatik und Mathematik, 60054 Frankfurt am Main, Germany}
\email{katthaen@math.uni-frankfurt.de}

\author{Julio Jos\'e Moyano-Fern\'andez}
\address{Universitat Jaume I, Campus de Riu Sec, Departamento de Matem\'aticas \& Institut Universitari de Matem\`atiques i Aplicacions de Castell\'o, 12071
Caste\-ll\'on de la Plana, Spain} \email{moyano@uji.es}

\author{Jan Uliczka}
\address{Universit\"at Osnabr\"uck, FB Mathematik/Informatik, 49069
Osnabr\"uck, Germany} \email{juliczka@uos.de}



\thanks{The first author was partially supported by  
the German Research Council DFG-GRK~1916. The second author was partially supported by the Spanish Government Ministerio de Econom\'ia y Competitividad (MINECO), grants MTM2012-36917-C03-03 and MTM2015-65764-C3-2-P, as well as by Universitat Jaume I, grant P1-1B2015-02.}

\begin{abstract}
Let $S$ be a multigraded polynomial ring such that the degree of each variable is a unit vector; so $S$ is the homogeneous coordinate ring of a product of projective spaces.
In this setting, we characterize the formal Laurent series which arise as Hilbert series of finitely generated $S$-modules.

Also we provide necessary conditions for a formal Laurent series to be the Hilbert series of a finitely generated module with a given depth.
In the bigraded case (corresponding to the product of two projective spaces), we completely classify the Hilbert series of finitely generated modules of positive depth.
\end{abstract}

\maketitle

\section{Introduction}
Let $\kk$ be a field. We consider the polynomial ring $R=\kk[X_1, \ldots , X_m]$ in $m$ indeterminates equipped with a $\ZZ^n$-grading, such that the degree of each variable is one of the unit vectors $\eb_i$ of $\ZZ^n$.
This setup includes the standard $\NN$-grading, as well as the \emph{fine grading}, where $m=n$ and $\deg{X_i}=\eb_i$ for $i \in [n]$.
\medskip

Let $M=\bigoplus_{\ab \in \ZZ^n} M_{\ab}$ be a finitely generated $\ZZ^n$-graded $R$-module.
The \demph{Hilbert series} of $M$ is the formal Laurent series
\[
H_M :=\sum_{\ab \in \ZZ^n} (\dim_{\kk} M_{\ab}) \cdot \tb^{\ab}  \in \ZZ(\!(t_1, \dotsc, t_n)\!),
\]
where we set $\tb^{\ab} := t_1^{a_1}\dotsm t_n^{a_n}$ for $\ab = (a_1, \dotsc, a_n)$.
This is well-defined, because the graded components $M_{a}$ of $M$ are finite-dimensional $\kk$-vector spaces, and, since $R$ is positively graded, there exists a $\bb \in \ZZ^n$ such that $M_{\ab}=0$ if $\ab \ngeq \bb$ (componentwise).
The Hilbert series is known to carry important information about $M$, for example its dimension or its multiplicity.
In the present work, we consider the following question:
\begin{question}\label{Q}
	Which formal Laurent series arise as Hilbert series of $R$-modules (in a certain class)?
\end{question}

An obvious property of Hilbert series is that all their coefficients are nonnegative.
If we allow non--finitely generated modules then this is already all that can be said: Any nonnegative series $H=\sum_{\ab} c_\ab \tb^{\ab}$ is the Hilbert series of the $R$--module  $\bigoplus_\ab \left(R/\mm\right)(-\ab)^{c_\ab}$, for $\mm=(X_1, \ldots , X_m)$.

So we restrict our attention to finitely generated modules.
This condition yields a second necessary condition for a series $H$ to be a Hilbert series:
It has to be a rational function with denominator $\prod_i (1-\tb^{\deg X_i})$.
In the standard-$\NN$-graded situation, these two conditions already characterize Hilbert series, as it was shown by the third author in \cite[Corollary 2.3]{U}.
However, they are not sufficient in the multigraded situation, cf. \Cref{ex:nonh}.

One of the main results of the present paper is a complete answer to Question \ref{Q} in the $\ZZ^n$-graded situation, see \Cref{thm:hilberts}.
This generalization is non\-trivial, as there are new phenomena in this setting.
While this result is somewhat technical, we obtain very satisfying specializations in the fine-graded and the bigraded situations, respectively; see \Cref{cor:e2} and \Cref{cor:bigraded}.
Also, we clarify the relation between Hilbert series and merely nonnegative series in \Cref{lem:ecken}.

It seems natural to further generalize these ideas to arbitrary multigradings.
However, in this generality, arithmetical issues arise.
For instance, there exists a formal Laurent series with integral coefficients which is not a Hilbert series, but after multiplication with $2$ it is, cf. \Cref{ex:arith};
the example already shows that one cannot hope for a characterization using linear inequalities in this setting.
In the present paper, we do not further pursue this direction.

\medskip
One of the difficulties of Question \ref{Q} is that if $H$ is a Hilbert series of some module, then any nonnegative series which coincides with $H$ in all but finitely many coefficients is also a Hilbert series (to see this, if $H = H_M$ for some module $M$, then one might replace finitely many components of $M$ in the lowest degrees by copies of $\kk$).
Thus it seems natural to rule this out, i.e. to consider modules which do not contain a copy of the residue field as a submodule.
Algebraically, this amounts to requiring that the depth of $M$ should be positive.
Generalizing this idea leads to the notion of \demph{Hilbert depth}, introduced by the third author in \cite{U}. 
Recall that the Hilbert depth\footnote{In the literature concerned with the Stanley depth, the term ``Hilbert depth'' refers to a different invariant, see for example Bruns, Krattenthaler and Uliczka~\cite{BKU}. The latter is defined via Hilbert decompositions and so it is sometimes called \emph{decomposition Hilbert depth}. In the standard-$\ZZ$-graded setup, these two notions coincide, but in general our Hilbert depth is only an upper bound for the decomposition Hilbert depth.} of a formal Laurent series $H$ is defined as
\[ 
\hdep(H) := \sup \set{ \dep(N) ~|~ N \text{  f.~g.~gr.~$R$--module with } H_N=H }.
\]
We consider the Hilbert depth only for those formal Laurent series which actually arise as Hilbert series of some module.
Again, in the standard-$\NN$-graded setting, Hilbert series of a given Hilbert depth have been classified in \cite{U}.
Our next main result is a quite general class of linear inequalities which are satisfied by every Hilbert series with a given Hilbert depth.
We formulate our result in terms of the projective dimension, but at least over the polynomial ring, this is equivalent to the depth via the Auslander-Buchsbaum formula.
\begin{restatable*}{theorem}{TorUngl}\label{thm:tor}
	Let $R = \bigoplus_\ab R_\ab$ be a (commutative Noetherian) $\ZZ^n$-graded $\kk$-algebra, such that $\dim_\kk R_0 < \infty$.
	Let further $p \in \NN$ and $M, N$ be finitely generated $R$-modules.
	If $\pdim M \leq p$ and $N$ is a $p$-th syzygy module, then the following inequality holds:
	\begin{equation}\label{eq:allgungl}
	\frac{H_M H_N}{H_R} \geq 0.
	\end{equation}
\end{restatable*}

In general, a classification of Hilbert series with a given Hilbert depth seems to be very difficult.
Therefore, we contempt ourselves with an important special case.
Our third main result is a complete characterization of the Hilbert series with positive Hilbert depth in the bigraded (i.e. $\ZZ^2$-graded) setup, see \Cref{th:stc}.
The condition that we obtain is similar to the general inequalities of \Cref{thm:tor}, but quite different in nature to the condition of \Cref{thm:hilberts}.

In the previous paper \cite{MU}, the second and the third author characterized Hilbert series of positive Hilbert depth over a bivariate polynomial ring $\kk[X,Y]$ endowed with a non-standard $\ZZ$-grading. In \Cref{thm:nonstandart} we show that this result can be restated in a way analogous to \Cref{th:stc}.

\medskip
\paragraph{\bf Related work}
The classical result of Macaulay \cite{Mac} answers Question \ref{Q} for \emph{cyclic} modules in the standard-$\NN$-graded situation.
This work has recently been extended by Boij and Smith in \cite{BS}.
These authors study Hilbert series in the standard-$\NN$-graded setup, with the additional assumption that only modules ge\-nerated in degree $0$ are considered.
The main difference to the present work, however, is that in \cite{BS} \emph{closure} of the set of Hilbert series is considered (with respect to a suitable topology on the space of formal Laurent series).

\section{Notation and preliminaries}

Let us fix some notation before continuing. We will use boldface letters $\ab, \ub, \gb, \dotsc$ to denote elements of $\ZZ^n$ or $\NN^n$.
For such a vector $\ab$, we write $a_i$ for its $i$-th component.

For $i \in [n]$ we denote by $\eb_i$ the $i$-th unit vector. Moreover, for $\db = (d_1, \dotsc, d_n) \in \NN^n$ and $\tb = (t_1, \dotsc, t_n)$, we define
	\begin{align*}
	\tb^\db &:= \prod_{i=1}^n t_i^{d_i} &\!\!\!\!\!\!\!\!\!\!\!\!\!\!\!\!\text{and}\\
	(1-\tb)^\db &:= \prod_{i=1}^n (1-t_i)^{d_i}.
	\end{align*}

For a formal Laurent series $H \in \ZZ(\!(t_1, \dotsc, t_n)\!)$ we write $\coeff{H}{\ab}$ for
the coefficient of $\tb^\ab$ in $H$.
We call $H$ \demph{nonnegative} if every coefficient of $H$ is nonnegative and we denote this by $H \geq 0$.

We consider the partial order on $\ZZ^n$ which is given by coordinatewise comparison.
In other words, for $\ub, \vb \in \ZZ^n$ we write $\ub \leq \vb$ if and only if  $u_i \leq v_i$ for all $i \in [n]$.
Moreover, we denote the coordinate minimum resp.~ maximum of $\ub$ and $\vb$ by $\ub\wedge\vb$ resp.~ $\ub\vee\vb$.

\subsection{Multivariate Hilbert Polynomials}
Let $H = \sum_i h_i t^i$ be a univariate formal Laurent series, such that $H\cdot (1-t)^d$ is a Laurent polynomial. Then it is well-known that there exist a polynomial $p \in \QQ[Z]$ and $i_0 \in \ZZ$, such that $h_i = p(i)$ for all $i \geq i_0$.
If $H$ is a Hilbert series, then $p$ is called the \demph{Hilbert polynomial}.

Here we give a multivariate generalization of this fact.
\begin{lemma}\label{lem:HPoly}
	Let $Q \in \ZZ[t_1^{\pm 1}, \dotsc, t_n^{\pm 1}]$ be a Laurent polynomial, $\db \in \NN^n$ and
	\[ H = \frac{Q}{(1-\tb)^{\db}}. \]
	Then there exist a $\ub \in \ZZ^n$ and a polynomial $p \in \ZZ[Z_1, \dotsc, Z_n]$ such that $\coeff{H}{\ab} = p(\ab)$ for $\ab \geq \ub$.
	Moreover, $p$ is uniquely determined (but $\ub$ is not), and $\deg_{Z_i} p \leq d_i - 1$ for $1 \leq i \leq n$.
\end{lemma}
\begin{proof}
	Write $Q = \sum_{\ub \in \Omega} c_\ub \tb^\ub$ for a suitable finite set $\Omega \subseteq \ZZ^n$.
	Recall the series expansion
	\[ \frac{t^u}{(1-t)^d} = \sum_{a \geq u} \binom{a-u+d-1}{d-1} t^i.\]
	From this, it is clear that
	\[ \coeff{H}{\ab} = \sum_{\ub \in \Omega} c_\ub \binom{a_1 - u_1 + d_1 -1 }{d_1-1}\dotsm \binom{a_n - u_n + d_n -1 }{d_n-1} \]
	for all $\ab \geq \bigvee_{\ub \in \Omega} \ub$.
	Hence $\coeff{H}{\ab}$ is given by a polynomial of the claimed degree.
	
	The uniqueness follows from the fact that the Zariski closure of $\NN^n$ inside $\mathbb{C}^n$ equals $\mathbb{C}^n$.
\end{proof}

\begin{definition}\label{def:hpoly}
	In the situation of \Cref{lem:HPoly}, we call $p$ the \demph{Hilbert polynomial} of $H$ and denote it with $\HP{H}$.
	In the degenerate case $n=0$, $H$ is an integer and we set $\HP{H} = H$.
\end{definition}

\section{Which series are Hilbert series?}

Let $\kk$ be a field. We consider the polynomial ring $R=\kk[X_1, \ldots , X_m]$ equipped with a $\ZZ^n$-grading, such that each variable is homogeneous.
In this section, we deal with the question of which formal Laurent series $H \in \ZZ(\!(t_1, \dotsc, t_n)\!)$ arise as Hilbert series of finitely generated $R$-modules.
There are two obvious necessary conditions:
\begin{itemize}[$\diamond$]
	\item $H \geq 0$ coefficientwise, and
	\item $\prod_{i=1}^n (1-\tb^{\deg X_i}) H$ is a Laurent polynomial.
\end{itemize}
It is a consequence of \cite[Theorem 2.1]{U} that in the case of the standard $\NN$-grading (i.e. $n = 1$), these conditions are already sufficient.
In general, this is not true, as we will see below.

In a previous work, the second and third author already obtained a general characterization of Hilbert series, cf. \cite[Corollary 2.2]{MU}:
\begin{theorem}\label{thm:hdec}
	A formal Laurent series $H \in \ZZ(\!(t_1, \dotsc, t_n)\!)$ is the Hilbert series of a finitely generated $R$-module if and only if it can be written in the form
	\begin{equation}\label{eq:hdec}
		H = \sum_{I \subseteq [m]} \frac{Q_I}{\prod_{i \in I}(1-\tb^{\deg X_i})}
	\end{equation}
	for Laurent polynomials $Q_I \in \ZZ[t_1^{\pm 1}, \dotsc, t_n^{\pm 1}]$ and $I \subseteq [m]$ with nonnegative coefficients.
\end{theorem}
A decomposition as in \Cref{eq:hdec} is called a \demph{Hilbert decomposition} of $H$.
This result is stated and proven in \cite{MU} for $\ZZ$-gradings only, but the proof given there can be easily extended for multigradings.

\begin{example}\label{ex:arith}
	Let $R = \kk[X_1, X_2, X_3]$ with the $\NN$-grading given by $\deg X_1 = 2, \deg X_2=3$ and $\deg X_3 = 5$.
	Consider the series
	\begin{align*}
	H &= \frac{1}{2}\left(\frac{t}{1-t^2}  + \frac{1}{1-t^3} + \frac{1+t}{1-t^5} + \frac{t^7}{(1-t^3)(1-t^5)}\right) \\
	&= \frac{t+t^3}{1-t^6} + \frac{1}{(1-t^5)(1-t^6)}
	\end{align*}
	One sees immediately that $2H$ has a Hilbert decomposition and therefore it is indeed a Hilbert series of a finitely generated $R$-module.
	In particular, $2H$ and thus $H$ satisfy the necessary conditions mentioned above.
	However, $H$ is not a Hilbert series, as it does not admit a Hilbert decomposition.
	
	To see this, note that $H$ has a pole of order $2$ at $t = 1$.
	Considering the possible summands in \Cref{eq:hdec}, it follows that $Q_{\set{2,3,5}} = 0$ and that at least one of $Q_{\set{2,3}}, Q_{\set{2,5}}$ and $Q_{\set{3,5}}$ is non-zero.
	One can compute that the $i$-th coefficient of $H$ is of the order $\frac{i}{30} + O(1)$.
	On the other hand, the $i$-th coefficient of $1/(1-t^a)(1-t^b)$ is of the order $\frac{1}{ab} + O(1)$ for coprime $a,b \in \NN$.
	As $2\cdot3,2\cdot5,3\cdot5 < 30$, the series $H$ does not have a Hilbert decomposition and thus it does not arise as a Hilbert series.
\end{example}

\subsection{The standard $\mathbb{Z}^n$-grading}

The criterion of \Cref{thm:hdec} is very useful for showing that a given Laurent series is a Hilbert series: One only needs to construct a Hilbert decomposition.
However, it is rather difficult to use this criterion to show that a given series is \emph{not} a Hilbert series. Moreover it does not provide a good insight into the structure of the set of Hilbert series.
We would like to have a characterization of the Hilbert series in terms of inequalities. 
In view of the preceding example there is no hope for such a characterization in full generality.

So we now specialize our considerations to the case that the degree of every variable of $R$ is a unit vector.
More precisely, we consider the case that $R = \kk[X_{ij} \with 1 \leq i \leq n, 1 \leq j \leq m_i]$, where $n \in \NN$, $\mb =  (m_1, \dotsc, m_n) \in \NN^n$ and $\deg X_{ij} = \eb_i$.
In this setting, we give a characterization of the Hilbert series of finitely generated modules over $R$ in terms of certain inequalities. 
Roughly speaking, this can be seen as an implicitization of the set of Hilbert series.
Before we can state our result we need to introduce some notation.

\begin{definition}
	Let $H = \sum_{\ab \in \ZZ^n} h_\ab \tb^\ab \in \ZZ(\!(t_1, \dotsc, t_n)\!)$ be a formal Laurent series.
	For $I \subseteq [n]$ and $\ub \in \ZZ^n$ we define
	\[ H|_{I,\ub} := \sum_{\ab \in \NN^I} h_{\ub + \ab} \tb^\ab \in \ZZ(\!(t_i | i \in I)\!),\]
	where $\NN^I := \set{\sum_{i \in I} c_i\eb_i \with c_i \in \NN} \subseteq \NN^n$.
	We call $H|_{I,\ub}$ the \demph{restriction of $H$} to $\ub + \NN^I$.
\end{definition}
	Note that $H|_{\emptyset,\ub} = \HP{H|_{\emptyset,\ub}}= \coeff{H}{\ub}$ for $\ub \in \ZZ^n$.
	Also, note that the summands of $\tb^\ub H|_{I,\ub}$ are summands of $H$. It is more convenient to consider $H|_{I,\ub}$ instead of $\tb^\ub H|_{I,\ub}$, because the former lives in a smaller ring of Laurent series.

\begin{definition}
	Let $p \in \QQ[Z_1, \dots, Z_n]$ be a polynomial.
	We call a monomial $\Zb^\rb$ appearing in $p$ \demph{extremal} if it does not divide any other monomial of $p$.
	Moreover, we say that $p$ has \demph{positive extremal coefficients} if the coefficient of every extremal monomial of $p$ is positive.
\end{definition}

The following characterization of Hilbert series of $R$-modules is the main result of this section.
\begin{theorem}\label{thm:hilberts}
	The following statements are equivalent for a formal Laurent series $H \in \ZZ(\!(t_1, \dotsc, t_n)\!)$:
	\begin{enumerate}
		\item There exists a finitely generated graded $R$-module $M$ whose Hilbert series equals $H$.
		\item $H$ satisfies the following two conditions:
		\begin{enumerate}
			\item $H \cdot \prod_{i=1}^n (1-t_i)^{m_i}$ is a polynomial, and
			\item for every $\ub \in \ZZ^n$ and every $I \subseteq [n]$, the Hilbert polynomial of the restriction $H|_{I,\ub}$ of $H$ has positive extremal coefficients.
		\end{enumerate}
	\end{enumerate}
\end{theorem}
\begin{remark}
	The condition that $H \geq 0$ is implicit in the last condition of \Cref{thm:hilberts} above, because $\HP{H|_{\emptyset, \ub}}= \coeff{H}{\ub}$ for $\ub \in \ZZ^n$.
\end{remark}

\begin{example}\label{ex:nonh}
	Let $n = 2$ and $\mb = (3,3)$.
	Consider the series
	\[ H := \sum_{i \geq 0}\sum_{j \geq 0} (i-j)^2 t_1^i t_2^j = \frac{t_1t_2^2 + t_1^2t_2 + t_1^2 - 6t_1t_2+t_2^2 + t_1+t_2}{(1-t_1)^3(1-t_2)^3}.\]
	Clearly $H \cdot \prod_{i=1}^2 (1-t_i)^{m_i}$ is a polynomial, and it is also clear that $H \geq 0$. So $H$ satisfies the obvious necessary conditions for being a Hilbert series.
	
	Moreover, $\HP{H} = (i-j)^2 = i^2 -2ij+j^2$.
	Here, all three monomials are extremal, so in particular $\HP{H}$ does not have positive extremal coefficients.
	Hence $H$ does not arise as Hilbert series of a finitely generated $R$-module.

Although this can be obtained from \Cref{thm:hilberts}, one way to see this directly is as follows:
Assume to the contrary that there $H = H_M$ for a finitely generated $R$-module $M$.
We write $\deg_1$ and $\deg_2$ for the first and second component of the degree of an element $m \in M$, respectively.
Let $g_1, \dotsc, g_r$ be a set of generators of $M$ and let $\deg g_k = (i_k,j_k)$.
If $i_k < j_k$, then $(R g_k)_{(j_k,j_k)} = 0$ and hence $\deg_1 m < j_k$ for any $m \in R g_k$.
Similarly, if $i_k > j_k$, then $\deg_2 m < i_k$ for any $m \in R g_k$. Hence, in both cases we have that $\min(\deg_1 m, \deg_2 m) \leq \max(i_k,j_k)$ for all $m \in R g_k$. As $M$ is generated by $g_1, \dotsc, g_r$, it follows that $\min(\deg_1 m, \deg_2 m)\leq \max(i_1,\dotsc,i_r, j_1,\dotsc,j_r)$ for all $m \in M$.
This contradicts our assumption that $H_M = H$.
\end{example}

\begin{example}
	Our next example shows that it is not sufficient to consider only the Hilbert polynomial of $H$.
	Let
	\[ H := \sum_{i \geq 0}\sum_{j \geq 0}\sum_{j \geq 0} \left( (i-j)^2 + ijk \right) t_1^i t_2^j t_3^k.\]
	It holds that $\HP{H} = (i-j)^2 + ijk$, so all extremal coefficients are nonnegative.
	On the other hand, for $I = \set{1,2}, \ub = 0$ it holds that $\HP{H|_{I,\ub}} =  (i-j)^2$ and this polynomial does not have positive extremal coefficients.
	Thus $H$ is not a Hilbert series.
\end{example}

One common trait in the theory of Hilbert series is that many properties can be determined by examining only those exponents which are below the exponent $\gb$ which is the join of the exponents of the numerator.
So one might hope to sharpen \Cref{thm:hilberts} by showing that one only needs to consider restrictions $H_{I, \ub}$ for $\ub \leq \gb$.
However, the next example shows that this does not hold.
\begin{example}
	For $\lambda \in \NN$, consider the series
	\[ H := \sum_{i \geq 0}\sum_{j \geq 0}\sum_{k \geq 0} \left( i^2 + j^2 + ij(k - \lambda)(k-\lambda-2) \right) t_1^i t_2^j t_3^k.
	\]
	This series is nonnegative, because for $i,j,k \in \NN$ it holds that
	\[ i^2 + j^2 + ij(k - \lambda)(k-\lambda-2) \geq i^2 + j^2 + ij\cdot (-1)  \geq 0. \]
	The Hilbert polynomial clearly has nonnegative extremal coefficients.
	Moreover, for $I = \set{2,3}$ and any $\ub \in \ZZ^n$, the Hilbert polynomial of the restriction is 
	\[ \HP{H|_{I,\ub}} = j^2 + jk^2 + \text{ lower terms}\]
	so it has nonnegative extremal coefficients. By symmetry, the same holds for $I = \set{1,3}$.
	Further, it follows from \Cref{lem:ecken} below that the Hilbert polynomials of restrictions $H|_{I,\ub}$ with $|I| = 1$ have nonnegative extremal coefficients. 
	
	So it remains to consider the case $I = \set{1,2}$. Let $\ub = (0,0, \kappa) \in \ZZ^3$. Then
	\[ \HP{H|_{I,\ub}} = i^2 + j^2 + (\kappa - \lambda)(\kappa-\lambda-2)ij \]
	and all three terms are extremal. So this restriction has nonnegative extremal coefficients if and only if $\kappa \neq \lambda + 1$. In particular, $H$ is not a Hilbert series.
	
	On the other hand, writing $H$ as a rational function one sees that the degrees of all terms in the numerator are less or equal than $\gb := (2,2,2) \in \NN^n$.
	Thus for $\lambda \geq 2$ it is not sufficient to consider restrictions $H|_{I,\ub}$ with $\ub \leq \gb$.
\end{example}

We need some preparations before we give the proof of \Cref{thm:hilberts}.
First, note that Hilbert decompositions are compatible with restrictions in the following sense:
\begin{lemma}
	Let $H \in \ZZ(\!(t_1, \dotsc, t_n)\!)$ be a formal Laurent series.
	If $H$ has a Hilbert decomposition, then so does every restriction $H|_{I, \ub}$ of $H$ for $\ub \in \ZZ^n, I \subseteq[n]$.
\end{lemma}
\begin{proof}
	If $H$ has a Hilbert decomposition, then there exists a finitely generated graded $R$-module $M = \bigoplus_\ab M_\ab$ with $H = H_M$.
	Set
	\[ M' := \bigoplus_{\ab \in \ub + \NN^I} M_\ab.\]
	This is a module over 
	$R' := \kk[X_{ij} \with i \in I] \subseteq R$
	in a natural way. We give it the structure of an $R$-module by letting the other variables act as zero.
	Its Hilbert series equals $\tb^\ub H|_{I, \ub}$, hence this series has a Hilbert decomposition. 
	But then $H|_{I, \ub}$ clearly has a Hilbert decomposition as well.
\end{proof}

Next, we show that polynomials with positive extremal coefficients admit a certain decomposition. This is the key step in our proof of \Cref{thm:hilberts}.
\begin{proposition}\label{prop:binomial}
	For a polynomial $p \in \QQ[Z_1, \dots, Z_n]$, the following statements are equivalent:
	\begin{enumerate}
		\item $p$ has positive extremal coefficients.
		\item $p$ can be written as follows:
		\[ p = \sum_{(\ab, \rb) \in \Omega} c_{(\ab,\rb)} \binom{Z_1 + r_{1} - a_{1}}{r_{1}}\dotsm \binom{Z_n + r_{n} - a_{n}}{r_{n}}\]
		for some finite set $\Omega \subseteq \ZZ^n \times \NN^n$ and $c_{(\ab,\rb)} > 0$ for all $(\ab,\rb)\in \Omega$.
	\end{enumerate}
	If, in addition, there exists a $\ub \in \ZZ^n$ such that $p(\ab) \in \ZZ$ for all $\ab \geq \ub$, then the coefficients $c_{(\ab,\rb)}$ can be chosen to be natural numbers. 
\end{proposition}
\begin{proof}
	We start with the implication (1) $\Rightarrow$ (2).
	Let $\Zb^\rb$, for $\rb \in \NN^n$, 
	be an extremal monomial of $p$, and let $c$ be its coefficient.
	For $\ab \in \ZZ^n$ consider
	\[ Q := \binom{Z_1+r_1-a_1}{r_1} \dotsm \binom{Z_n+r_n-a_n}{r_n}. \]
	It is easy to see that $\Zb^\rb$ is also an extremal monomial of $Q$, and in fact it is its only extremal monomial. Further, its coefficient $1/r_1!\dotsm r_n! > 0$.
	Hence, the corresponding terms cancel in 
	\[ p_1 := p - \frac{c}{r_1!\dotsm r_n!} Q.\]
	We show that $p_1$ still satisfies the hypothesis (1), so the claim follows by induction.
	For this, note that the only possible new extremal monomials of $p_1$ are the monomials $\Zb^\rb / Z_i$ 
	for $1 \leq i \leq n$, so we need to compute their coefficients in $Q$.
	We start with one factor of $Q$:
	\begin{align*}
		\binom{Z + r - a}{r} &= \frac{1}{r!}(Z+r-a)(Z-r-a-1)\dotsm(Z+1-a)\\
		&= \frac{1}{r!}\left( Z^r + \left(\sum_{\ell=0}^{r-1}r-a-\ell\right) Z^{r-1} + \dotsb \right)\\
		&= \frac{1}{r!}\left( Z^r + r\left(\frac{r+1}{2}-a\right) Z^{r-1} + \dotsb \right).
	\end{align*}
	This implies that
	\[
	\begin{split}
	Q = \frac{1}{r_1!\dotsm r_n!}\left( \Zb^\rb 
	+ \sum_{i=1}^n r_i\left(\frac{r_i+1}{2}-a_i\right) \frac{\Zb^\rb}{Z_i}
	+\text{lower terms}\right).
	\end{split}
	\]
	Note that for a sufficiently large choice of $a_1, \dotsc, a_n$, the coefficients of $\Zb^\rb / Z_i$ 
	become arbitrarily negative. Hence, for large $\ab$, $p_1$ still satisfies the hypothesis.
	
	For the other implication, note that if $p$ can be written as in (2), then there can be no cancellation between the extremal monomials on the right-hand side. So the coefficients of the extremal monomials of $p$ are (sums of) multiples of the $c_i$, and thus positive.
	
	Finally, assume that $p(\ab) \in \ZZ$ for all $\ab \in \ZZ^n$ which are greater than some fixed $\ub$.
	We first note that this implies that $p(\ab) \in \ZZ$ for all $\ab \in \ZZ^n$, cf. Corollary I.1.2 and Corollary XI.1.5 in Cahen and Chabert \cite{CCbook}.
	So by a classical result of Ostrowski \cite{ostrowski} (see also \cite[Corollary XI.1.11]{CCbook}), $p$ can be written as an integral linear combination of polynomials of the form 
	\[ H_{k_1, \dots, k_n} := \binom{Z_1}{k_1}\dotsm\binom{Z_n}{k_n} \]
	with $k_1, \dotsc, k_n \in \NN$.
	If $\Zb^\rb$ is an extremal monomial of $p$, then only $H_{r_1, \dots, r_n}$ contributes to this term, so its coefficient $c$ is a multiple of the corresponding coefficient in $H_{r_1, \dots, r_n}$, which is $1/r_1!\dotsm r_n!$.
	It now follows from the construction above that the $c_{(\ab,\rb)}$ are positive integers.
\end{proof}

\begin{proof}[Proof of \Cref{thm:hilberts}]
	We start with the necessity:
	If $H$ has a Hilbert decomposition, then so does every restriction of $H$.
	Hence we only need to consider the case $I = [n]$ and thus we need to show that $\HP{H}$ has positive extremal coefficients.
	Consider a Hilbert decomposition
	\[ H = \sum_{(\ab, \eb) \in \Omega} \frac{c_{(\ab,\eb)} \tb^\ab}{(1-\tb)^\eb} \]
	of $H$, where $\Omega \subseteq \ZZ^n \times \NN^n$ is a suitable finite index set and $c_{(\ab,\eb)} \geq 0$ for all $(\ab, \eb) \in \Omega$.
	Expanding every summand into a series, it follows that
	\[ \HP{H} = \sum_{(\ab, \eb) \in \Omega} c_{(\ab,\eb)} \binom{Z_1 + e_{1} - a_{1}-1}{e_{1}-1}\dotsm \binom{Z_n + e_{n} - a_{n}-1}{e_{n}-1}. \]
	Hence $\HP{H}$ has positive extremal coefficients by \Cref{prop:binomial}.
	\medskip
	
	Now we turn to the sufficiency:
	We proceed by induction over the number of variables of $H$, with the base case being trivial.
	First, assume that $\HP{H} \neq 0$.
	By assumption, its extremal coefficients are nonnegative, so \Cref{prop:binomial} yields a decomposition
	\[
	\HP{H} = \sum_{(\ab,\rb)\in\Omega} c_{(\ab,\rb)} \binom{Z_1 + r_{1} - a_{1}}{r_{1}}\dotsm \binom{Z_n + r_{n} - a_{n}}{r_{n}}
	\]
	for a finite set $\Omega \in \ZZ^n\times \NN^n$ and $c_{(\ab,\rb)} > 0$ for all $(\ab,\rb)\in\Omega$.
	Moreover, $\HP{H}$ takes integer values on large $\ab \gg 0$, so $c_{(\ab,\rb)} \in \NN$ for all $(\ab,\rb)\in\Omega$.
	Let 
	\[ H_1 := \sum_{(\ab,\rb)\in\Omega} \frac{c_{(\ab,\rb)}\tb^{\ab}}{\prod_i (1-t_i)^{r_i+1}}. \]
	It is easy to see that $\HP{H_1} = \HP{H}$, so there exists a $\gb \in \ZZ^n$ such that $\coeff{H}{\ab} = \coeff{H_1}{\ab}$ for all $\ab\geq \gb$. 
	
	Set $H' := H - \tb^\gb H_1|_{[n], \gb}$. It holds that $\HP{H'} = 0$ and we claim that $H'$ still satisfies the hypothesis on the extremal coefficients.
	To see this, consider $I \subseteq [n]$ and $\ub \in \ZZ^{n}$.
	If $\ub + \NN^I \cap \gb + \NN^n = \emptyset$, then $H'|_{I,\ub} = H|_{I,\ub}$;
	otherwise, let $\vb \in \ub + \NN^I \cap \gb + \NN^n$. Then $\coeff{H'}{\ab} = 0$ for all 
	$\ab \geq \vb$ and hence $\HP{H'|_{I,\ub}} = 0$. In both cases, the hypothesis is satisfied.
	
	Next, we consider the case that $\HP{H} = 0$.
	In this case, the exponent vectors of the nonzero terms of $H$ are contained in finitely many translates of coordinate hyperplanes. Hence we may decompose $H$ as a sum of series in $n-1$ variables as follows:
	Choose $\gb \in \ZZ^n$ such that $\coeff{H}{\ab} = 0$ for all $\ab \geq \gb$.
	For $1 \leq i \leq n$ and $0 \leq j \leq g_i-1$ let $\ub(i,j) := (g_1, \dotsc, g_{i-1},j,0,\dotsc,0) \in \ZZ^n$.
	We decompose $H$ as follows:
	\[ H = \sum_{i=1}^n \sum_{j=0}^{g_{i}-1} \tb^{\ub(i,j)} H|_{[n] \setminus \{i\}, \ub(i,j)} \]
	Every restriction of $H$ is a series in at most $n-1$ variables, so the claim follows by induction.
\end{proof}

\subsection{Non-Negative series}
Our next goal is to clarify the relation between series satisfying the hypothesis of \Cref{thm:hilberts} and series which are merely nonnegative.
We will need the following convex geometric lemma.
\begin{lemma}\label{lem:polytope}
	Let $P \subseteq \RR^n$ be a polytope and $v \in P$ a vertex, such that $v + u \notin P$ for all $u \in \RR_{\geq 0}^n, u \neq 0$.
	Then there exists a linear form $\sigma \in (\RR^n)^*$ which attains its maximum over $P$ exactly at $v$ and whose coefficients are nonnegative integers.
\end{lemma}
\begin{proof}
	Let $P' \subseteq P$ be the convex hull of all vertices of $P$ which are different from $v$.
	Further, let $Q := v + \RR_{\geq 0}^n$. Both $P'$ and $Q$ are convex sets, and our assumption implies that $P' \cap Q = \emptyset$.
	Then there exists a separating hyperplane, i.e. a linear form $\sigma \in (\RR^n)^*$ such that
	\begin{equation}\label{eq:sep}
		\max(\sigma(p) \with p \in P') < \inf(\sigma(q) \with q \in Q).
	\end{equation}
	We may assume that $\sigma$ has rational coefficients, and after clearing denominators we may even assume that the coefficients of $\sigma$ are integers.
	We show that $\sigma$ has the claimed properties.
	It is clear that the maximum of $\sigma$ over $P$ is attained only at $v$.
	To see that the coefficients of $\sigma$ are nonnegative assume that $\sigma(\eb_i) < 0$ for some unit vector $\eb_i$. Then $\sigma(v + \lambda \eb_i)$ can be arbitrarily negative for large $\lambda \gg 0$, contradicting \eqref{eq:sep}.
\end{proof}

The following lemma shows that if $H \geq 0$ then some extremal coefficients are automatically nonnegative.
\begin{theorem}\label{lem:ecken}
	Let $H \in \ZZ(\!(t_1, \dotsc, t_n)\!)$ be a formal Laurent series,
	such that $H \cdot (1-\tb)^\mb$ is a Laurent polynomial for some $\mb \in \NN^n$.
	The following conditions are equivalent:
	\begin{enumerate}
		\item $H \geq 0$.
		\item For all $\ub \in \ZZ^n$ and $I \subseteq [n]$, every extremal monomial of $\HP{H|_{\ub,I}}$ which is also a vertex of its Newton polytope has a positive coefficient.
	\end{enumerate}
\end{theorem}
\begin{proof}
	The implication (2) $\Rightarrow$ (1) is clear, because $\coeff{H}{\ab} = \HP{H|_{\ab, \emptyset}}$ for all $\ab \in \ZZ^n$.
	So we only need to show the other implication.
	
	Let $\Zb^\rb$ be an extremal monomial of $p := \HP{H|_{\ub,I}}$ which is also a vertex of its Newton polytope.
	By \Cref{lem:polytope}, there exists a linear form $\sigma \in (\RR^n)^*$ with nonnegative integral coefficients, which attains its maximum over the Newton polytope exactly at $\rb$.
	Consider the linear map $\hat{\sigma}: \ZZ[Z_1, \dotsc, Z_n] \to \ZZ[Z]$, given by $\Zb^\rb \mapsto Z^{\sigma(\rb)}$.
	Then $\hat{\sigma}(p)$ is a univariate polynomial, which attains nonnegative values at sufficiently large integers.
	Hence its leading coefficient is nonnegative.
	On the other hand, by our choice of $\sigma$, this leading coefficient of $\hat{\sigma}(p)$ equals the coefficient of $\Zb^\rb$ in $p$.
	So the claim is proven.
\end{proof}

\subsection{The fine graded case}

If there are at most two variables with the same degree, then the obvious necessary conditions for Hilbert series are also sufficient. This includes in particular the case of the fine graded polynomial ring.
\begin{corollary}\label{cor:e2}
	In the situation of \Cref{thm:hilberts}, assume that $m_i \leq 2$ for all $i$.
	
	Then the following two statements are equivalent for a formal Laurent series $H \in \ZZ(\!(t_1, \dotsc, t_n)\!)$:
	\begin{enumerate}
		\item There exists a finitely generated graded $S$-module $M$ whose Hilbert series equals $H$.
		\item $H$ satisfies the following two conditions:
		\begin{enumerate}
			\item $H \geq 0$, and
			\item $H \cdot \prod_{i=1}^n (1-t_i)^{m_i}$ is a polynomial.
		\end{enumerate}
	\end{enumerate}
\end{corollary}
\begin{proof}
	For $I \subseteq [n]$ and $\ub \in \ZZ^n$ let $p = \HP{H|_{I,\ub}}$.
	The hypothesis that $m_i \leq 2$ for all $i$ implies that every monomial of $p$ is squarefree.
	Hence its Newton polytope is a 0/1-polytope, so every lattice point in it is a vertex.
	In particular, all extremal monomials of $p$ are vertices of its Newton polytope, so the claim follows from \Cref{lem:ecken}.
\end{proof}

\section{General inequalities for the Hilbert depth}
In this section, we present a class of linear inequalities for the Hilbert series of modules with a given depth.
We relax our assumptions on $R$ and allow it to
be an arbitrary (commutative) $\kk$-algebra with a positive $\ZZ^n$-grading, such that $\dim_\kk R_0 < \infty$.
The general idea is to compare the Hilbert series in question with \emph{all} Hilbert series of modules from a certain class.
Recall that an $R$-module $N$ is called \emph{$p$-th syzygy module} if it can be realized as the $p$-th syzygy of some $R$-module $N'$.
See \cite{B} for alternative characterizations of syzygy modules.
\medskip

The following is the main result of this section.
\TorUngl
\begin{proof}
	First note that every free $R$-module $F$ satisfies $H_{F \otimes N} = \frac{H_F H_N}{H_R}$.
	This clearly holds for $F = R(a)$ with $a \in \ZZ^n$, and it is easily seen that the equality is preserved under direct sums.
	
	Consider a free resolution of $M$:
	\[ \FF: 0 \to F_p \to \dotsb \to F_0 \to 0. \]
	We compute that
	\[ \sum_{i \geq 0} (-1)^i H_{F_i \otimes N} = \sum_{i \geq 0} (-1)^i \frac{H_{F_i} H_N}{H_R} = \frac{H_N}{H_R}  \sum_{i \geq 0} (-1)^i H_{F_i} = \frac{ H_N}{H_R} H_M. \]
	On the other hand, it holds that
	\[ \sum_{i \geq 0} (-1)^i H_{F_i \otimes N} =  \sum_{i \geq 0} (-1)^i H_{H_i(\FF \otimes N)} = \sum_{i \geq 0} (-1)^i H_{\Tor_i(M,N)}. \]
	Let $N'$ be an $R$-module such that $N$ is the $p$-th syzygy module of $N'$.
	Then it holds for $i > 0$ that $\Tor^R_i(M,N) = \Tor^R_{i+p}(M, N') = 0$
	because $\pdim M \leq p$.
	We conclude that
	\[ \frac{H_M H_N}{H_R} = \sum_{i \geq 0} (-1)^i H_{\Tor_i(M,N)} = H_{M \otimes N} \geq 0. \qedhere\]
\end{proof}

\begin{remark}
	Let us consider some extremal cases of this theorem in the case that $R$ is the polynomial ring (with an arbitrary $\ZZ^n$-grading).
	\begin{enumerate}[(1)]
		\item If $p = \dim R$, all $p$-th syzygies modules are free. Hence \eqref{eq:allgungl} reduces to the statement that $H_M \geq 0$ for every $R$-module $M$.
		\item For $p = 0$, every module with $\pdim M \leq p$ is free. On the other hand, let $M$ be a module satisfying the inequality \eqref{eq:allgungl} for all $0$-th syzygy modules $N$.
		Choosing $N = \kk = R / (X_1, \dotsc, X_m)$ yields that $H_M$ is of the form
		\[ \frac{Q}{\prod_{i=1}^n (1-t_i)} \]
		for some $Q \in \NN[t_1^{\pm 1}, \dotsc, t_n^{\pm 1}]$. If $Q = \sum_{\ab \in \ZZ^n} q_{\ab} \tb^{\ab}$, 
		then the free module 
		\[ M' := \bigoplus_{\ab \in \ZZ^n} R(-\ab)^{q_{\ab}} \]
		has the same Hilbert series as $M$.
		Hence, \eqref{eq:allgungl} exactly describes the Hilbert series of free modules. 
	\end{enumerate}
\end{remark}

In general, the inequalities \eqref{eq:allgungl} are not sufficient for a Hilbert series to have a given Hilbert depth.
Nevertheless, in the next two sections, we consider two special situations where slightly stronger inequalities are indeed sufficient.

For later use, we also record a useful criterion for Hilbert series of modules of positive depth.
	\begin{proposition}\label{prop:hdep}
		Let $R$ be a $\ZZ^n$-graded polynomial ring, such that every variable is homogeneous.
		Then a formal Laurent series $H \in \ZZ(\!(t_1, \dotsc, t_n)\!)$ is the Hilbert series of a finitely generated $R$-module of positive depth if and only if it can be written in the form
		\begin{equation}\label{eq:hdec0}
		H = \sum_{\substack{I \subseteq [m]\\I \neq \emptyset}} \frac{Q_I}{\prod_{i \in I}(1-\tb^{\deg X_i})}
		\end{equation}
		for Laurent polynomials $Q_I \in \ZZ[t_1^{\pm 1}, \dotsc, t_n^{\pm 1}], I \subseteq [m]$ with nonnegative coefficients.
	\end{proposition}
The difference to \Cref{thm:hdec} is that there is no term $Q_\emptyset$. We call a Hilbert decomposition as in \Cref{eq:hdec0} a \demph{Hilbert decomposition without polynomial part}.
As in \Cref{thm:hdec}, this result is essentially contained in \cite[Prop.~2.4]{MU}, but there it is stated only for $\ZZ$-graded rings; the proof in our context follows by the same method.

\section{The Bigraded case}
In this section we consider the $\ZZ^2$--graded situation.
More precisely, let $R=\mathbb{K}[X_1,\ldots , X_m,Y_1,\ldots , Y_{\tilde{m}}]$
be the polynomial ring with a $\ZZ^n$-grading given by $\deg X_i =(1,0)$ for all $i \in [m]$ and $\deg Y_i =(0,1)$ for all $i \in [\tilde{m}]$.

Specializing \Cref{thm:hilberts} to this situation we obtain the following characterization of Hilbert series over $R$:
\begin{proposition}\label{cor:bigraded}
	For a formal Laurent series $H \in \ZZ(\!(t_1, t_2)\!)$, there exists a finitely generated graded $R$-module $M$ with $H = H_M$ if and only if $H$ satisfies the following conditions:
	\begin{enumerate}
		\item $H \geq 0$,
		\item $H\cdot (1-t_1)^{m}\cdot(1-t_2)^{\tilde{m}}$ is a Laurent polynomial, and
		\item $\HP{H}$ has positive extremal coefficients.
	\end{enumerate}
\end{proposition}
\begin{proof}
	The necessity of the conditions is clear from \Cref{thm:hilberts}.
	
	For the sufficiency, we show that given conditions imply the hypotheses of \Cref{thm:hilberts}.
	So consider $I \subseteq [2]$ and $\ub \in \ZZ^2$. If $I = \emptyset$, then $\HP{H|_{I,\ub}} = \coeff{H}{\ub} \geq 0$, and $\HP{H}$ has positive extremal coefficients by assumption.
	
	So we only need to consider the case that $|I| = 1$.
	In this case $\HP{H|_{I,\ub}}$ is a univariate polynomial, so its only extremal monomial is the leading one, which is also a vertex of its Newton polytope. Thus the claim follows from \Cref{lem:ecken}.
\end{proof}

Our next main result will be the characterization of Hilbert series of positive Hilbert depth over $R$. For this we will consider certain pairs of sequences of indices in $\ZZ^2$. 

\begin{definition}
\begin{enumerate}
\item A sequence $\left(\ub^{(i)}\right)_{i=1}^p$ in $\ZZ^2$ is called \emph{declining}, if
\[ 
	u_1^{(i)} < u_1^{(i+1)} \quad \mbox{and} \quad u_2^{(i)} >u_2^{(i+1)} \quad \mbox{for} \quad i=1, \ldots , p-1.
\]
The set of all declining sequences will be denoted by $\ST$.
\item For a sequence $U=\left(\ub^{(i)}\right)_{i=1}^p$ we define $\down{U} := (\ub^{(i)} \wedge \ub^{(i+1)})_{i=1}^{p-1}$.
\end{enumerate}
\end{definition}
The two sequences $U$ and $\down{U}$ can be visualized as a \qq{staircase}, see \Cref{fig:staircase}.
\begin{figure}[h]
	\begin{tikzpicture}[scale=0.6]
		\draw[dotted] (0,0) grid [step=1cm](8,6);
		\draw[] (0,0) -- (8,0);
		\draw[] (0,0) -- (0,6);
		\draw[thick] (1,5) -- (1,4) -- (3,4) -- (3,3) -- (4,3) -- (4,1) -- (7,1);
		\draw[fill] (1,5) circle [radius=0.1]; 
		\draw[fill] (3,4) circle [radius=0.1]; 
		\draw[fill] (4,3) circle [radius=0.1]; 
		\draw[fill] (7,1) circle [radius=0.1]; 
		\draw[fill=white] (1,4) circle [radius=0.1];
		\draw[fill=white] (3,3) circle [radius=0.1];  
		\draw[fill=white] (4,1) circle [radius=0.1]; 
	\end{tikzpicture}
\caption{A declining sequence $U=\left\{(1,5),(3,4),(4,3),(7,1)\right\}$ with $\down{U}=\left\{(1,4) ,(3,3),(4,1)\right\}$.}
\label{fig:staircase}
\end{figure}

\begin{theorem}\label{th:stc}
	Let $H \in \ZZ(\!(t_1, t_2)\!)$ be a formal Laurent series, which is the Hilbert series of some finitely generated graded $R$-module.
	Let further $S := \kk[X, Y]$ be the standard--$\ZZ^2$--graded polynomial ring in two variables.
	Then the following statements are equivalent:
	\begin{enumerate}[(a)]
		\item $H$ has positive Hilbert depth.
		\item For any finitely generated torsionfree $S$-module $N$, it holds that
		\begin{equation}\label{eq:HR}
		\frac{H\cdot H_N}{H_S} \geq 0.
		\end{equation}
		\item Condition (b) holds for any finitely generated torsionfree $S$ module of rank $1$, i.e. every fractional monomial ideal $I \subseteq \kk[X^{\pm 1},Y^{\pm 1}]$.
		\item $H= \sum_{i,j} h_{ij} \,t_1^i t_2^j$ satisfies the condition
		\begin{equation} \label{eq:st}
		\sum_{(i,j) \in \down{U}} h_{ij} \leq \sum_{(i,j) \in U} h_{ij} \quad \text{for all} \ \ U \in \ST. \tag{ST}
		\end{equation}
	\end{enumerate}
\end{theorem}

\begin{remark}
\begin{enumerate}[1)]
	\item Note that the torsionfree modules (over a domain) are exactly the first syzygy modules.
		
	\item If $(\ub^{(i)})_{i=1}^1$ is a declining sequence with only one entry, then \eqref{eq:st} asserts that $h_{\ub^{(1)}} \geq 0$.
	Hence any series satisfying condition \eqref{eq:st} is in particular nonnega\-tive.

	\item Although condition \eqref{eq:st} resembles its counterpart in the non-standard--$\ZZ$-graded case, namely the condition \eqref{eq:star} of \Cref{thm:nonstandart} below, there are important differences: 
		The inequalities required by condition \eqref{eq:star} only relate coefficients lying within one common period of the module's Hilbert function, they have the same number of terms on both sides, and this number is bounded above by \[\min \{\deg X, \deg Y \}.\]
		By contrast, a declining sequence $U$ may have any number of entries, which may be arbitrarily separated, and the right-hand side of \eqref{eq:st} has always one term more than the left-hand side.
		One might think that it could be possible to weaken condition \eqref{eq:st} by restricting it to a subset of $\ST$ consisting of somehow bounded sequences, but this turns out to be a vain hope.
		For instance, the examples
			\[ H_k = 1 + \sum_{i=0}^k \frac{t_1^i}{1-t_2}, \quad k  \in \NN \]
		show that we cannot afford to restrict condition \eqref{eq:st} to those sequences $U=(\ub^{(i)})_{i=0}^p$ where $\max_{i \in [m-1]} \{ u_1^{(i+1)} - u_1^{(i)} \} \leq C$ for some $C \in \NN$. 
\end{enumerate}
\end{remark}

Before we present the proof of \Cref{th:stc} we give several lemmata.
First of all, it turns out to be convenient to consider a slightly larger class of inequalities.
For this, we call a sequence $\left(\ub^{(i)} \right)_{i=0}^p$ in $\ZZ^2$ \emph{weakly declining}, if 
$ u_1^{(i)} \leq u_1^{(i+1)}$ and $u_2^{(i)} \geq u_2^{(i+1)}$ for $i=0, \ldots, p-1$
and let $\STp$ denote the set of all weakly declining sequences.
Moreover, for a formal Laurent series $H = \sum_{\ub \in \ZZ^2} h_\ub \tb^\ub \in \ZZ(\!(\tb)\!)$ and a sequence $U$, we set
\[ \sigma_U(H) := \sum_{\ub \in U} h_\ub - \sum_{\ub \in \down{U}} h_\ub \]

\begin{lemma}\label{lem:weakly}
	Let $H\in \ZZ(\!(t_1,t_2)\!)$ be a formal Laurent series.
	Then $\sigma_U(H) \geq 0$ for all $U \in \ST$ if and only if $\sigma_U(H) \geq 0$ for all $U \in \STp$.
\end{lemma}
\begin{proof}
	One implication is trivial.
	So assume that $\sigma_U(H) \geq 0$ for all $U \in \ST$ and consider a weakly declining sequence $U=\left(\ub^{(i)}\right)_{i=0}^p$ in $\ZZ^2$.
	If $u_1^{(i)} = u_1^{(i+1)}$ for some index $i$, then $\ub^{(i+1)}= \ub^{(i)} \wedge \ub^{(i+1)}$.
	Hence if $U'$ is the sequence obtained from $U$ by deleting $\ub^{(i+1)}$, then $\sigma_U(H) = \sigma_{U'}(H)$.
	Similarly, if $u_1^{(i)} = u_1^{(i+1)}$ then $\ub^{(i)} =\ub^{(i)} \wedge \ub^{(i+1)}$, so we may delete $\ub^{(i)}$ from $U$.
	After finitely many such deletions, we obtain a declining sequence $U'' \in \ST$ with $\sigma_U(H) = \sigma_{U''}(H)$.
	As $\sigma_{U''}(H) \geq 0$ by assumption, the claim follows.
\end{proof}

The following lemma essentially reduces the question to the fine-graded situation.

Recall that $\coeff{H}{\ub}$ denotes the coefficient of $\tb^\ub$ in the series $H$.
\begin{lemma}\label{lem:dec}
	Assume that $H \in \NN(\!(t_1,t_2)\!)$ admits a Hilbert decomposition without polynomial part.
	Then there exists an $\Nb_0 \in \ZZ^2$ with the following property:
	
	For any $\Nb \geq \Nb_0$, there exists a decomposition $H = H_1 + H_2$
	such that
	\begin{enumerate}
		\item $H_1$ is of the form $\frac{Q_1}{1-t_1} + \frac{Q_2}{1-t_2}$ for some $Q_1, Q_2 \in \NN[t_1^{\pm 1}, t_2^{\pm 1}]$.
		Moreover, it holds that $\coeff{H_1}{\Nb} = 0$  and $\coeff{H_1}{\ub \wedge \Nb} = \coeff{H_1}{\ub}$ for all $\ub \in \ZZ^2$.
		\item $H_2$ has a Hilbert decomposition without polynomial part and it satisfies $\coeff{H_2}{\ub} = 0$ for all $\ub < \Nb$.
	\end{enumerate}
\end{lemma}
\begin{proof}
Choose a Hilbert decomposition of $H$ without polynomial part:
\[ H = \sum_{\ab\in\NN^2, \ab \neq 0} \frac{Q_\ab}{\prod_{i =1}^2(1-t_i)^{a_i}}. \]
Choose $\Nb_0 \in \ZZ^2$ which is strictly larger than the degrees of all monomials in the numerator polynomials $Q_{(1,0)}$ and $Q_{(0,1)}$.
By repeatedly using the relation
\[
	\frac{1}{\prod_{i=1}^2(1-t_i)^{a_i}} = t_j\frac{1}{\prod_{i=1}^2(1-t_i)^{a_i}}  + (1-t_j) \frac{1}{\prod_{i =1}^2(1-t_i)^{a_i}},
\]
we may modify the Hilbert decomposition such that it satisfies the following:
\begin{itemize}
\item[$\diamond$] $Q_{(0,0)} = 0$.
\item[$\diamond$] For every $\ab \in \NN^2, \ab \neq (0,1),(1,0)$ such that $Q_\ab \neq 0$, it holds that $Q_\ab$ contains no monomials of degree strictly less than $\Nb$.
\item[$\diamond$] The polynomial $(1-t_1) Q_{(0,1)} + (1-t_2)Q_{(1,0)}$ contains only monomials which are strictly less than $\Nb$.
\end{itemize}
We set 
\[ H_1 := \frac{Q_{(1,0)}}{1-t_1} + \frac{Q_{(0,1)}}{1-t_2} \]
and $H_2 := H - H_1$.
The claimed properties of $H_1$ and $H_2$ follow readily.
\end{proof}

The third lemma is the key step in our proof of \Cref{th:stc}.
Here we show how to decompose a non-negative Laurent series of the form $Q_0 + Q_1/(1-t_1) + Q_2/(1-t_2)$.
\begin{lemma} \label{l:r1}
	Let
		\[ H= \sum_{i=k}^{\infty} \sum_{j= \ell}^{\infty} h_{ij} \, t_1^i t_2^j \]
	be a formal Laurent series satisfying condition \eqref{eq:st}.
	Assume further that  $h_{p \ell} = 0$ for some $p \geq k$ and let $\mu := \min \set{ h_{i \ell} \with i > p }$.
	Then the series
		\[ H - \frac{1}{1-t_2} \sum_{i=k}^{p-1} h_{i \ell} t_1^i t_2^\ell  - \frac{1}{1-t_1} \mu t_1^{p+1} t_2^{\ell}  \]
	satisfies condition \eqref{eq:st} as well.
\end{lemma}

\begin{proof}
To prove the lemma, it is enough to show the following two claims:
\begin{enumerate}
	\item For any $k \leq r < p$ the series 
	\[ \tilde{H}:= H - h_{r \ell} \, \frac{t_1^r t_2^{\ell}}{1-t_2} \]
	satisfies condition \eqref{eq:st}.

	\item Further, for $\mu := \min \set{ h_{i \ell} \with i > p }$ the series 
	\[ \hat{H}:= H - \mu \, \frac{  t_1^{p+1} t_2^{\ell}}{1-t_1} \]
	satisfies condition \eqref{eq:st}.
\end{enumerate}

We start with the proof of the first claim.
For this we have to show that condition \eqref{eq:st} is valid for $\tilde{H} =: \sum_{i,j} \tilde{h}_{ij} \, t_1^i t_2^j$.
The coefficients $\tilde{h}_{ij}$ and $h_{ij}$ are equal except for $i = r$ and $j \geq \ell$, where $\tilde{h}_{ij} = h_{ij}-h_{r \ell} $ holds.
Hence we only have to consider declining sequences $U \subseteq\ST$ such that $U$ or $\down{U}$ intersect $C_{r,\ell}:=\left\{ (r,j) ~|~ j \geq \ell \right\}$. 
If $ C_{r,\ell}$ intersects both $U$ and $\down{U}$, 
then we have
\[ \sigma_U(\tilde{H}) = \sigma_U(H) - h_{r\ell} + h_{r\ell} = \sigma_U(H) \geq 0 \]
so we may assume that only $U$ intersects with $ C_{r,\ell}$.
In this case it is not clear a priori whether the corresponding inequality still holds for $\tilde{H}$, because only the right--hand side of the original inequality is diminished.
(A similar problem occurs in \cite{MU}, where the analogous inequalities are called \emph{critical}.)

Let $\ub^{(e)}=(r, j')$ denote the intersection of $U$ and $ C_{r,\ell}$.
Since we assume that $\down{U}$ and $ C_{r,\ell}$ do not intersect, the sequence $U$ either ends in $\ub^{(e)}$ or the point $\ub^{(e)} \wedge \ub^{(e+1)}=\left(r,j'' \right)$, and hence all further points of $U$ and $\down{U}$ lie in the half--plane $\{(x,y) ~|~ y < \ell \}$.
Since all coefficients of $H$ and $\tilde{H}$ in this half--plane vanish, we may assume that the staircase ends in $\ub^{(e)}$.
We amend $U$ by $\tilde{\ub}^{(e+1)} := (p,\ell)$ to build $\tilde{U}$. 
Note that $\tilde{U}$ is weakly declining, hence $\sigma_{\tilde{U}}(H) \geq 0$ by \Cref{lem:weakly}. 
As $\ub^{(e)} = (r,j')$, it follows that $\ub^{(e)}\wedge \tilde{\ub}^{(e+1)} = (r,\ell)$ and thus
\[ 
	\sigma_{U}(\tilde{H}) =  \sigma_{U}(H) - h_{r\ell} =  (\sigma_{\tilde{U}}(H) + h_{r\ell} - h_{p\ell}) - h_{r\ell} = \sigma_{\tilde{U}}(H)  \geq 0.
\]

Now we turn to the proof of the second claim.
By the choice of $\mu$ the series $\hat{H} =: \sum_{i,j} \hat{h}_{ij} t_1^i t_2^j$ is nonnegative.
Similar to the proof of the first claim the verification of condition \eqref{eq:st} for 
$\hat{H}$ reduces to an inspection of sequences $U\in\ST$ intersecting $B_{p+1,\ell} := \left\{ (i,\ell) ~|~ i \geq p+1 \right\}$ in $U$, but not in $D$.
Let $\ub^{(e)} = (i',\ell)$ be this intersection; 
again we may assume that $U$ ends in $\ub^{(e)}$.

Let $\tilde{U}$ be the new weakly declining sequence obtained from $U$ by replacing the last element $\ub^{(e)} = (i', \ell)$ by $\tilde{\ub}^{(e)} := (p,\ell)$.
We assumed that $\ub^{(e-1)} \wedge \ub^{(e)} \notin B_{p+1,\ell}$, hence it holds that $u_1^{(e-1)} \leq p$ and thus $\ub^{(e-1)} \wedge \ub^{(e)} = \ub^{(e-1)}\wedge \tilde{\ub}^{(e)}$.
 It follows that
\begin{align*}
\sigma_{U}(\hat{H}) &= \sigma_{\tilde{U}}(\hat{H}) - h_{p\ell} + \coeff{\hat{H}}{(i',\ell)} = \sigma_{\tilde{U}}(H) + \coeff{\hat{H}}{(i',\ell)} \\
&=  \sigma_{\tilde{U}}(H) +(h_{i'\ell} - \mu) \geq 0,
\end{align*}
since $h_{i'\ell} - \mu \geq 0$.
\end{proof}

\noindent Now we are able to prove Theorem \ref{th:stc}:
 
\begin{proof}[Proof of Theorem \ref{th:stc}] \
\begin{asparadesc}
\item[(a) $\Rightarrow$ (b):]
	Let $N$ be a finite torsionfree $S$-module.
	After shifting the degrees, we may assume that all homogeneous components of $N$ have non-negative degrees.
	Let $\ub \in \ZZ^2$ be a multidegree.
	By \Cref{lem:dec} we can find a decomposition $H = H_1 + H_2$, such that all coefficients of $H_2$ in degrees less or equal than $\ub$ vanish.
	Hence 
\[ \mathfrak{c}\left(\frac{H\cdot H_N}{H_S},\ub\right) = \mathfrak{c}\left(\frac{H_1\cdot H_N}{H_S},\ub \right ). \]
	But $\frac{H_1}{H_S}$ is a sum of terms of the form $(1-t_1)\tb^\ab$ or $(1-t_2)\tb^\ab$ with $\ab \in \ZZ^2$ and nonnegative coefficients.
	Hence it is enough to show that $(1-t_1) H_N \geq 0$ and $(1-t_2) H_N \geq 0$.
	Note that because $N$ is torsionfree, multiplication by $X$ and $Y$ gives inclusions $N(-(1,0)) \hookrightarrow N$ and $N(-(0,1)) \hookrightarrow N$, thus these claims follow.
	
\item[(b) $\Rightarrow$ (c):] This is trivial.
\item[(c) $\Rightarrow$ (d):]
	For a given declining sequence $U = (\ub^{(i)})_{i=1}^p \in \ST$, consider the fractional ideal $I$ generated by $X^{-u_1^{(1)}}Y^{-u_2^{(1)}}, \dotsc, X^{-u_1^{(p)}}Y^{-u_2^{(p)}}$.
	Considering the minimal free resolution of $I$ over $S$ yields that
	\[
	\frac{H_I}{H_S} = \sum_{i=1}^p \tb^{-\ub^{(i)}} - \sum_{i = 1}^{p-1} \tb^{(-\ub^{(i)}) \vee (-\ub^{(i+1)})}.
	\]
	This and (c) imply that
	\begin{align*}
	0 \leq \frac{H\cdot H_I}{H_S} &= \left(\sum_{\ab \in \ZZ^2} h_{\ab} \tb^\ab\right) \left(\sum_{i=1}^p \tb^{-\ub^{(i)}} - \sum_{i = 1}^{p-1} \tb^{(-\ub^{(i)}) \vee (-\ub^{(i+1)})}\right)\\
	&= \sum_{\ab \in \ZZ^2} \tb^\ab \left(\sum_{i=1}^p h_{\ab + \ub^{(i)}} - \sum_{i = 1}^{p-1} h_{\ab - ((-\ub^{(i)}) \vee (-\ub^{(i+1)}))}\right) \\
	&= \sum_{\ab \in \ZZ^2} \tb^\ab \left(\sum_{i=1}^p h_{\ab + \ub^{(i)}} - \sum_{i = 1}^{p-1} h_{(\ab + \ub^{(i)}) \wedge (\ab + \ub^{(i+1)})}\right). 
	\end{align*}
	Note that the coefficient of $\tb^{\mathbf{0}}$ equals $\sigma_U(H)$, hence the latter is nonnegative.
	So $H$ satisfies \eqref{eq:st}. 
	
\item[(d) $\Rightarrow$ (a):]
	First, choose a decomposition $H = P + H'$, such that $P$ is a polynomial with non-negative coefficients and $H'$ has a Hilbert decomposition with no polynomial part.
	Let $\tilde{\Nb} \in \ZZ^2$ such that all non-zero coefficients of $P$ lie in degrees below $\tilde{\Nb}$.
	Again, using \Cref{lem:dec} there is $\Nb \geq \tilde{\Nb}$ and a decomposition $H' = H_1 + H_2$, such that $H_1$ and $H_2$ satisfy the conditions mentioned above.
	
	We are going to construct a Hilbert decomposition without polynomial part of $H_3 := P + H_1$.
	As $H_2$ already has such a Hilbert decomposition, this is enough to prove the claim.
	For this, we first need to show that $H_3$ still satisfies \eqref{eq:st}.
	
	Let $U = (\ub^{(i)})_{i=0}^p \in \ST$ be a declining sequence.
	Then $U' := (\ub^{(i)} \wedge \Nb)_{i=0}^p$ is a weakly declining sequence. 
	If $\Nb \in U'$, then it is easy to see that
	\[ \sigma_{U'}(H_3) = \coeff{H_3}{\Nb} = 0. \]
	Otherwise, our choice of $\Nb$ and \Cref{lem:dec} imply that $\coeff{H_3}{\ab \wedge \Nb} = \coeff{H_3}{\ab}$ for all $\ab \in \ZZ^2$.
	Hence it follows that $\sigma_{U}(H_3) = \sigma_{U'}(H_3)$.
	Further, it holds that 
	\[ \sigma_{U'}(H_3) = \sigma_{U'}(H_3 + H_2) \geq 0, \]
	because the coefficients of $H_2$ vanish below $\Nb$ and $H = H_2 + H_3$ satisfies \eqref{eq:st}.
	Hence $H_3$ satisfies \eqref{eq:st} as well.
	The series $H_3$ is of the form
	\[ Q_0 + \frac{Q_1}{1-t_1} + \frac{Q_2}{1-t_2}. \]
	We obtain a Hilbert decomposition of $H_3$ without polynomial part by repeatedly applying \Cref{l:r1}.
\end{asparadesc}
\end{proof}

With a little more work, one can show directly that condition (d) implies (c). Indeed, we already showed that every declining sequence gives rise to a fractional ideal. On the other hand, to every fractional ideal one can associate a declining sequence in a natural way, thus proving the equivalence directly.

\section{The non-standard \texorpdfstring{$\ZZ$}{Z}-graded case}
\newcommand{\sg}{\langle \alpha,\beta \rangle}
\newcommand{\fc}{\mathcal{F}_{\alpha,\beta}}
Let $R = \kk[X,Y]$ with the grading given by $\deg X = \alpha$ and $\deg Y = \beta$ for two coprime numbers $\alpha, \beta \in \NN$.
In \cite{MU}, the second and third author characterized the Hilbert series of modules of positive depth over this ring.
In this section, we give a reformulation of this result along the lines of the previous results.

First of all, we recall the characterization of Hilbert series of finitely generated graded modules over this ring (cf.~the authors in \cite[Theorem 2.6]{KMU}): These are exactly the formal Laurent series $H$ with nonnegative integral coefficients, such that $H\cdot (1-t^\alpha)(1-t^\beta)$ is a Laurent polynomial, thus there is no analog of condition (3) of \Cref{cor:bigraded} in this setting.

In order to state the characterization of Hilbert series with positive Hilbert depth we need some more notation.
Denote by $\sg \subseteq \NN$ the numerical semigroup generated by $\alpha$ and $\beta$.
A \demph{fundamental couple} (with respect to $\alpha$ and $\beta$) is pair $[I,J]$ of two integer sequences $I = (i_k)_{k=0}^m$ and $J = (j_k)_{k=0}^m$ which are subject to the following conditions:
\begin{enumerate}
	\item[(0)] $i_0=0$;
	\item[(1)] $i_1, \ldots , i_m, j_1, \ldots , j_{m-1} \in \NN\setminus\sg$ and $j_0,j_m \leq \alpha \beta$;
	\item[]
	\item[(2)] $\begin{array}{lllll}
	i_k \equiv j_k &\!\!\!\mod \alpha &~\mbox{and}~& i_k < j_k & ~~\mbox{for} ~~ k= 0, \ldots ,m;\\
	j_k \equiv i_{k+1} &\!\!\!\mod \beta &~\mbox{and}~& j_k > i_{k+1} & ~~\mbox{for} ~~ k= 0, \ldots ,m-1;\\
	j_m \equiv i_{0}  &\!\!\!\mod \beta &~\mbox{and}~& j_m \geq i_{0}. &
	\end{array}$
	\item[]
	\item[(3)] $|i_k - i_{\ell}| \in \NN\setminus\sg ~~\mbox{for} ~~1 \leq k < \ell \leq m$.
\end{enumerate}
Denote by $\fc$ the set of all fundamental couples. 
\medskip

The characterization of positive Hilbert depth over $R$ can be stated as follows:

\begin{theorem}\label{thm:nonstandart}
	Let $H \in \ZZ(\!(t)\!)$ be a formal Laurent series, which is the Hilbert series of some finitely generated graded $R$-module.
	Let further  $S := R / (X^\beta - Y^\alpha)$.
	Then the following statements are equivalent:
	\begin{enumerate}[(a)]
		\item $H$ hat positive Hilbert depth.
		\item For any finitely generated torsionfree $S$-module $N$, it holds that
		\begin{equation}\label{eq:HR2}
		\frac{H\cdot H_N}{H_R} \geq 0.
		\end{equation}
		\item Condition (b) holds for any finitely generated torsionfree $S$ module of rank $1$. 
		\item $H = \sum_{i} h_{i} t^i$ satisfies the condition
		\begin{equation}\label{eq:star}
		\sum_{i \in I} h_{i+n} \leq \sum_{j \in J} h_{j+n} ~\text{for all }n \in \NN, [I,J] \in \fc. \tag{$\star$}
		\end{equation}
	\end{enumerate}
\end{theorem}

\begin{remark}
\begin{asparaenum}
	\item The equivalence of (a) and (d) is the main theorem of \cite{MU}. This is also the most difficult part of the result.
	\item Note that the denominator of \Cref{eq:HR2} is $H_R$, in contrast to the $H_S$ in the denominator of \Cref{eq:HR}.
	In both cases, it is easy to see that $\frac{H_M H_N}{H_R} \geq 0$ implies that $\frac{H_M H_N}{H_S} \geq 0$. We present two examples to show that this implication is strict in both cases:
	\begin{enumerate}
		\item In the setting of \Cref{th:stc}, let $R = \kk[X_1,X_2,Y_1,Y_2]$, $N = S$ and $H_M = 1 /(1-t_1)$. Then $M$ clearly has Hilbert depth $1$, but
		\[ \frac{H_M H_N }{H_R} =\frac{1}{1-t_1} \cdot\frac{1}{(1-t_1)(1-t_2)} \cdot (1-t_1)^2(1-t_2)^2 = 1-t_2 \ngeq 0. \]
		\item In the setting of \Cref{thm:nonstandart}, let $\alpha = 2, \beta = 3$ and consider the Hilbert series $H_M = 1 + t^3$. The module $M$ has finite length and therefore Hilbert depth $0$, but for any torsionfree $S$-module $N$ it holds that
		\[ \frac{H_M H_N}{H_S} = (1+t^3)\cdot \frac{(1-t^3)(1-t^3)}{1-t^6}H_N = (1-t^2) H_N \geq 0, \]
		where for the last inequality, we use that multiplication by $X$ gives an injection $N(-2) \hookrightarrow N$.
	\end{enumerate}
\end{asparaenum}
\end{remark}

\noindent We need the following result about the structure of fundamental couples.
\begin{lemma}\label{lemma:fundcouple}
	Let $[I,J]$ be a fundamental couple of length $m$.
	Then there exist two integer sequences
	\begin{align*}
	\beta &> a_0 > a_1 > \dotsb > a_m = 0 \text{ and}\\
	0 &= b_0 < b_1 < \dotsb < b_m < \alpha
	\end{align*}
	such that
	\begin{equation}\label{eq:pres}
	\begin{aligned}
	i_k &= \alpha\beta - a_{k-1} \alpha - b_k \beta &\text{for } 1 \leq k \leq m &\text{ and}\\
	j_k &= \alpha\beta - a_k \alpha - b_k \beta &\text{for } 0 \leq k \leq m
	\end{aligned}
	\end{equation}
\end{lemma}
\begin{proof}
	Recall from Rosales, Garc\'ia-S\'anchez, Garc\'ia-Garc\'ia and Jim\'enez-Madrid \cite[Lemma 1]{R} that an integer $e \in \ZZ$ is not contained in $\sg$ if and only if there are $a,b \in \NN$ with $a,b \geq 1$ such that $e = \alpha\beta - a \alpha - b \beta$.
	In this case, $a$ and $b$ are uniquely determined and we denote them by $a(e)$ and $b(e)$, respectively.
	
	Now consider a fundamental couple $[I,J]$.
	We define $a_{k} := a(j_k)$ and $b_k := b(j_k)$ for $1 \leq k \leq m-1$. 
	Further, $j_0 \equiv 0 \mod{\alpha}$ so it can be written as $j_0 = \alpha\beta - \alpha a_0$ for some $0 \leq a_0 < \beta$. We set $b_0 = 0$, so $j_0 = \alpha\beta - a_{0} \alpha - b_0 \beta$ as required.
	Similarly, $j_m \equiv 0 \mod{\beta}$, so there exists $0 \leq b_m < \alpha$ such that $j_m = \alpha\beta - a_{m} \alpha - b_m \beta$ with $a_m = 0$.
	
	Next, set $\tilde{\imath}_k := \alpha\beta - a_{k-1} \alpha - b_k \beta$ for $1 \leq k \leq m$.
	We need to show that $i_k = \tilde{\imath}_k$. For this, note that $i_k \equiv j_k \equiv \tilde{\imath}_k \mod{\alpha}$ and similarly $i_k \equiv \tilde{\imath}_k \mod{\beta}$.
	Hence $i_k = \tilde{\imath}_k + r \alpha \beta$ for some $r \in \ZZ$.
	As $\tilde{\imath}_k < \alpha\beta$, it follows that $\tilde{\imath}_k + r\alpha \beta < 0$ for $r < 0$, so we may assume that $r \geq 0$.
	On the other hand, it $r > 0$, then
	\begin{align*}
	i_k &= \tilde{\imath_k} + r \alpha \beta 
	= \alpha\beta - a_{k-1} \alpha - b_k \beta  + r \alpha \beta\\
	&= \alpha\beta - a_{k} \alpha - b_k \beta  + \alpha(r \beta - a_{k-1} + a_k) > j_k,
	\end{align*}
	which is a contradiction. Thus $r = 0$ and so $i_k = \tilde{\imath}_k$ for $1 \leq k \leq m$.
	
	Finally, $i_k < j_k, j_{k-1}$ implies that $a_k < a_{k-1}$ and $b_k > b_{k-1}$ for $1 \leq k \leq m$.
\end{proof}

\begin{proof}[Proof of \Cref{thm:nonstandart}]
\begin{asparadesc}
	\item[(a) $\Rightarrow$ (b):]
		If $H$ has positive Hilbert depth, then it can be written as
		\[ H = \frac{Q_1}{1-t^\alpha} + \frac{Q_2}{1-t^\beta} + \frac{Q_{1,2}}{(1-t^\alpha)(1-t^\beta)}. \]
		with $Q_1, Q_2, Q_{1,2} \in \NN[t^{\pm 1}]$.
		Using that $N$ is torsionfree, it easily follows that \Cref{eq:HR2} holds for each summand and therefore for $H$.

	\item[(b) $\Rightarrow$ (c):] This is trivial.
	
	\item[(c) $\Rightarrow$ (d):]
		Let $[I,J]$ be a fundamental couple.
		Recall that $S = \kk[t^\alpha, t^\beta]$ is the monoid algebra of $\sg$.
		Let $N \subseteq \kk[t]$ be the $S$-module generated by $t^{\alpha\beta - j_0}, \dotsc, t^{\alpha\beta - j_m}$.
		This module is torsionfree, hence $\frac{H_M H_N}{H_R} \geq 0$ by assumption.
		To see that this inequality implies \Cref{eq:star}, we need to compute the Hilbert series of $N$.
		
		Let $(a_k)_{k=0}^m, (b_k)_{k=0}^m$ be the sequences as in Lemma \ref{lemma:fundcouple} and let 
		\[\tilde{N} := (X^{a_0}Y^{b_0}, \dotsc, X^{a_m}Y^{b_m}).\]
		It is easy to see that $\tilde{N}$ is the preimage of $N$ under the projection $R \to S$. 
		In particular, note that $X^\beta - Y^\alpha \in \tilde{N}$, because $X^{a_0}, Y^{b_m} \in \tilde{N}$.
		Hence $N \cong \tilde{N} / (X^\beta - Y^\alpha)$ and thus $H_N = H_{\tilde{N}} - t^{\alpha\beta} H_R$.
		
		By considering the minimal free resolution of $\tilde{N}$, one sees that its syzygies are generated in the degrees $a_{k-1} \alpha + b_{k} \beta$ for $1 \leq k \leq m$. 
		So we can compute the Hilbert series of $N$ as follows:
		\begin{align*}
		\frac{H_N}{H_R} &= \frac{H_{\tilde{N}} - t^{\alpha\beta} H_R}{H_R} 
		=\sum_{k=0}^m t^{a_k\alpha + b_k\beta} - \sum_{k=1}^m t^{a_{k-1} \alpha + b_k \beta} - t^{\alpha \beta}\\
		&=\sum_{k=0}^m t^{\alpha\beta - j_k} - \sum_{k=1}^m t^{\alpha\beta - i_k} - t^{\alpha\beta - i_0}
		=t^{\alpha\beta}\left(\sum_{j \in J} t^{-j} - \sum_{i \in I} t^{-i}\right)
		\end{align*}
		Together with \Cref{eq:HR2}, we obtain the following:
		\begin{align*}
		0 \leq \frac{H\cdot H_N}{H_R}  &= (\sum_{n \in \ZZ} h_n t^n) t^{\alpha\beta}\left(\sum_{j \in J} t^{-j} - \sum_{i \in I} t^{-i}\right) \\
		&= t^{\alpha\beta} \sum_{n \in \ZZ} t^n \left(\sum_{j \in J} h_{n+j} - \sum_{i \in I} h_{n+i}\right)
		\end{align*}
		So \Cref{eq:star} is satisfied for $[I,J]$.
		
	\item[(d) $\Rightarrow$ (a):] This is Theorem 3.13 of \cite{MU}.
\end{asparadesc}
\end{proof}

\section*{Acknowledgments}

The second author wishes to express his gratitude to the Institute of Mathematics at the University of Osnabr\"uck for kind hospitality.

\bibliographystyle{amsplain}
\bibliography{Girlanden_v18}

\end{document}